\documentclass{amsart} %[12pt]

\usepackage{verbatim, amssymb, enumerate}

\newcommand \la{\lambda}
\newcommand \bc{\mathbb{C}}
\newcommand \br{\mathbb{R}}
\newcommand \bbf{\mathbb{F}}
\newcommand \Rn{\mathbb{R}^n}
\newcommand \Cn{\mathbb{C}^n}
\newcommand \Fn{\mathbb{F}^n}
\newcommand \act{\mathcal{R}}
\newcommand \cS{\mathcal{S}}
\newcommand \<{\langle}
\renewcommand \>{\rangle}
\newcommand \ip {\<\cdot,\cdot\>}

\newcommand \id {\mathrm{id}}
\newcommand \ve {\varepsilon}

\theoremstyle{plane}
\newtheorem{theorem}{Theorem}
\newtheorem*{theorem*}{Theorem}
\newtheorem*{corollary*}{Corollary}
\newtheorem*{conj*}{Conjecture}
\newtheorem{lemma}{Lemma}

\newtheorem*{prop*}{Proposition}

\theoremstyle{definition}
\newtheorem{definition}{Definition}
\newtheorem*{definition*}{Definition}

\theoremstyle{remark}

\begin{document}

\title{The Duality principle for Osserman algebraic curvature tensors}

\author{Y.Nikolayevsky}
\address{Department of Mathematics and Statistics, La Trobe University, Victoria, 3086, Australia}
\email{y.nikolayevsky@latrobe.edu.au}

\author{Z. Raki\'{c}}
\address{Faculty of Mathematics, University of Belgrade, Serbia}
\email{zrakic@matf.bg.ac.rs}

\date{\today}

\thanks{Y.Nikolayevsky is partially supported by ARC Discovery grant DP130103485. Z.Raki\'{c} is partially supported by the Serbian Ministry of Education and Science, project No. 174012.}


\begin{abstract}
We prove that for an algebraic curvature tensor on a pseudo-Euclidean space, the Jordan-Osserman condition implies the Raki\'{c} duality principle, and that the Osserman condition and the duality principle are equivalent in the diagonalisable case.
\end{abstract}

\maketitle

\section{Introduction}
\label{s:intro}

Throughout the paper, $\bbf$ is the field $\br$ of real numbers or the field $\bc$ of complex numbers. Let $\ip$ be an inner product (a non-degenerate symmetric bilinear form) on $\Fn$, and let $\|\cdot\|^2$ be the associated quadratic form. When $\bbf=\br$, we will additionally distinguish between the Riemannian and the pseudo-Riemannian case depending on the signature of $\ip$. We say that $X \in \Fn$ is \emph{null}, if $\|X\|^2=0$.

An \emph{algebraic curvature tensor} $\act$ on $(\Fn, \ip)$ is a $(3,1)$ tensor having the same algebraic symmetries as that of the curvature tensor at a point of a Riemannian space. In more detail, $\act$ is a three-linear map, $(X,Y,Z) \mapsto \act(X,Y)Z \in \Fn$ for $X,Y,Z \in \Fn$ such that
\begin{equation}\label{eq:actdef}
\begin{gathered}
\act(X,Y)Z=-\act(Y,X)Z, \\
\act(X,Y)Z+\act(Y,Z)X+\act(Z,X)Y=0, \\
\<\act(X,Y)Z,V\>=\<\act(Z,V)X,Y\>.
\end{gathered}
\end{equation}

Given an algebraic curvature tensor $\act$ on $(\Fn, \ip)$ we denote $\act_X$, for $X \in \Fn$, the corresponding \emph{Jacobi operator} defined by $\act_X(Y) = \act(Y,X)X$ for $Y \in \Fn$. The Jacobi operator $\act_X$ is self-adjoint relative to the inner product $\ip$.

An algebraic curvature tensor $\act$ on  $(\Fn, \ip)$ is said to be \emph{Osserman}, if the eigenvalues of the Jacobi operator $\act_X$ and their multiplicities depend only on $\|X\|^2$, for $X \in \Fn$. In what follows it will be more convenient to adopt the following, equivalent definition. Denote $\chi_X=\chi_X(t)$ the characteristic polynomial of $\act_X$. The coefficient $f_j(X)$ of the term $t^j, \; 0 \le j \le n$, in $\chi_X(t)$ is a homogeneous polynomial of degree $2(n-j)$ of the coordinates of $X$. We say that $\act$ is \emph{Osserman}, if $f_j(X)=a_j \|X\|^{2(n-j)}$, for all $0 \le j \le n$, where $a_j \in \bbf$:

\begin{definition}\label{d:oss}
An algebraic curvature tensor $\act$ on $(\Fn, \ip)$ is said to be \emph{Osserman}, if there exists $P \in \bbf[t,y]$ such that $\chi_X(t)=P(t,\|X\|^2)$, for all $X \in \Fn$.
\end{definition}

In the real case, in the Riemannian signature, the eigenvalues of $\act_X$ and their multiplicities completely determine the conjugacy class of $\act_X$. This is no longer true in the indefinite signature, so one can separately define timelike and spacelike Osserman algebraic curvature tensors in the obvious way, but these two properties are easily seen to be equivalent. The picture is even more involved: the Jacobi operator does not need to be $\br$- or even $\bc$-diagonalisable and can have a nontrivial nilpotent part. To capture that finer structure, an algebraic curvature tensor $\act$ on $(\Fn, \ip)$ is said to be \emph{Jordan-Osserman}, if it is Osserman and for all non-null $X \in \Fn$, the Jordan normal form (the number and the sizes of the Jordan blocks) is the same. Note that the Jacobi operator of a Jordan-Osserman algebraic curvature tensor on $\Rn$ is allowed to have non-real eigenvalues. There is an extensive literature on the Osserman property, see \cite{Gil} and the references therein.

The motivation to study the Osserman property comes from Riemannian Geometry. Given a Riemannian manifold $M^n$, let $R$ be its curvature tensor and $R_X$ be the corresponding Jacobi operator. The manifold $M^n$ is said to be \emph{pointwise Osserman} if $R$ is Osserman at every point $x \in M^n$, and \emph{globally Osserman} if the eigenvalues of $R_X$ and their multiplicities are constant, for all $X$ in the unit tangent bundle of $M^n$. Locally two-point homogeneous spaces are globally Osserman, since the isometry group of each of those spaces is transitive on its unit tangent bundle. Osserman \cite{O} conjectured that the converse is also true. At present, the Osserman Conjecture is almost completely resolved by the results of Chi \cite{C}, for dimensions $n\neq 4k, \; k>1$, and $n=4$, and the first named author \cite{N1, N2, N3}, in all the remaining cases, with the few exceptions in dimension $n=16$. One of the crucial steps in the known proofs of the Osserman Conjecture is the following \emph{Raki\'c duality principle} \cite{R}:

\begin{quote}
\emph{Suppose $\mathcal{R}$ is an Osserman algebraic curvature tensor and $X, Y$ are unit vectors. Then $Y$ is an eigenvector of $\mathcal{R}_X$ if and only if $X$ is an eigenvector of $\mathcal{R}_Y$ (with the same eigenvalue).}
\end{quote}

The duality principle in the pseudo-Riemannian settings is introduced in \cite{AR1}, where it is extensively studied for diagonalisable pseudo-Riemannian manifolds and also proved for $n=4$. Recently, for algebraic curvature tensors in Riemannian signature, it was shown that the duality principle is in fact \emph{equivalent} to the Osserman property (in \cite{BM} for $n \le 4$, and in \cite{NR} for all $n$). In \cite{AR2} this equivalence was extended to the Lorentzian signature, and it was also shown that the duality principle holds for algebraic curvature tensors $\act$ with the Clifford structure (all such $\act$ are Osserman). Further results on the duality principle can be found in \cite{A1, A2}.

In this paper we study the relationship between the duality principle and the Osserman property in more general settings. We adopt the following definition. \begin{definition}\label{d:dua}
An algebraic curvature tensor $\act$ on $(\Fn, \ip)$ is said to satisfy the \emph{duality principle}, if for any non-null $X \in \Fn$ and any eigenvector $Y$ of $\act_X$, the vector $X$ is an eigenvector of $\act_Y$.
\end{definition}

The reason for excluding null vectors is explained in \cite[Remark~4]{AR1} and \cite[Section~6]{AR2}. We have the following theorem.

\begin{theorem}\label{t:jor}
A Jordan-Osserman algebraic curvature tensor on $(\Fn, \ip)$ satisfies the duality principle.
\end{theorem}

Theorem~\ref{t:jor} admits a partial converse. For an algebraic curvature tensor $\act$ on $\Fn$ denote $s_\act$ the set of those $X \in \Fn$ for which $\act_X$ is diagonalisable over $\bbf$.

\begin{definition}\label{d:semis}
An algebraic curvature tensor $\act$ on $\Fn$ is said to be \emph{semisimple}, if the set $s_\act$ has a nonempty interior in $\Fn$. Denote $\cS$ the set of semisimple algebraic curvature tensors on $\Fn$.
\end{definition}

Clearly, in the real case, in the Riemannian signature, any algebraic curvature tensor is semisimple. By the result of \cite[Theorem~1.2]{GI}, a real Jordan-Osserman algebraic curvature tensor in the signature $(p,q)$ is semisimple, unless $p=q$.

In general, it is not hard to see and will be shown in Section~\ref{s:eq} that for a semisimple $\act$, the Osserman and the Jordan-Osserman conditions are equivalent, so by Theorem~\ref{t:jor}, a semisimple Osserman algebraic curvature tensor satisfies the duality principle. The following theorem establishes the converse.

\begin{theorem}\label{t:rdp}
Suppose $\act$ is a semisimple algebraic curvature tensor on $(\Fn, \ip)$. Then $\act$ is Osserman if and only if it satisfies the duality principle.
\end{theorem}

\section{Proofs of the Theorems}
\label{s:eq}

Let $\act$ be an algebraic curvature tensor on $(\Fn, \ip)$. We call $X \in \Fn$ \emph{generic}, if there is a neighbourhood $U$ of $X$ in $\Fn$ such that for all $Y \in U$, the number and the sizes of the Jordan blocks in the Jordan normal form and the number of distinct eigenvalues of $\act_Y$ is the same (where in the case $\bbf=\br$ we also consider complex eigenvalues).

Note that if $\act$ is Jordan-Osserman, then any non-null $X \in \Fn$ is generic, by definition. The following lemma justifies our terminology.

\begin{lemma}\label{l:gene}
For an arbitrary algebraic curvature tensor $\act$ on $(\Fn, \ip)$, the set of generic vectors is open and dense in $\Fn$.
\end{lemma}
\begin{proof}
We first consider the complex case. Let $\act$ be an algebraic curvature tensor on $(\Cn, \ip)$. For $X \in \Cn$ denote $p(X)$ the number of different eigenvalues of the Jacobi operator $\act_X$, and denote $p=\max \{p(X) \, | \ X \in \Cn\}$. As the coefficients of the characteristic polynomial $\chi_X(t)$ of $\act_X$ are polynomials in the coordinates $(x_1, \dots, x_n)$ of $X$, the set of those $X \in \Cn$ for which $p(X)=p$ is a Zariski open subset of $\Cn$. For $X \in \Cn$, the $k$-th invariant factor of $\act_X, \; k=1, \dots, n$, is defined to be the greatest common divisor over $\bc[t]$ of the $k \times k$ minors of the matrix $\act_X-t\id$. All such minors (denote them $m^{(k)}_1(t; X), \dots, m^{(k)}_N(t; X), \; N=\binom{n}{k}^2$) are polynomials in $t$ whose coefficients are polynomials in the coordinates $(x_1, \dots, x_n)$ of $X$. We can consider them as polynomials of $t$ with the coefficients in the field of rational functions $\mathbf{F}=\bc(x_1, \dots, x_n)$, and then, by the repeated use of the Euclid's algorithm, find their greatest common divisor $I^{(k)}(t; X)$ over $\mathbf{F}[t]$. All $m^{(k)}_i(t; X), \; i=1, \dots, N$, are divisible by $I^{(k)}(t; X)$ over $\mathbf{F}[t]$, and hence for every fixed $X_0 \in \Cn$ outside the union of the zero sets of a finite number of polynomials in $(x_1, \dots, x_n)$ (the denominators and the leading coefficients of $I^{(k)}$ and of the $m^{(k)}_i$'s), all the polynomials $m^{(k)}_i(t; X_0)$ are divisible by $I^{(k)}(t; X_0)$ over $\bc[t]$. Furthermore, since the polynomials $h^{(k)}_i = m^{(k)}_i/I^{(k)}, \; i=1, \dots, N$, are coprime over $\mathbf{F}[t]$, there exist $a_1, \dots, a_N \in \mathbf{F}[t]$ such that $\sum_{i=1}^N a_i h^{(k)}_i = 1$. It follows that for every fixed $X_0 \in \Cn$ outside the union of the zero sets of a finite number of polynomials in $(x_1, \dots, x_n)$ (the denominators and the leading coefficients of $I^{(k)}$ and of the $m^{(k)}_i$'s and $a_i$'s), the polynomials $h^{(k)}_i(t; X_0)$ are coprime over $\bc[t]$. Therefore, for all $X$ from a nonempty, Zariski open subset $S \subset \Cn$, the $k$-th invariant polynomial of $\act_{X}$ is $I^{(k)}(t; X)$. Then for all $X \in S$, the elementary divisors of $\act_{X}$ are given by $I^{(k)}(t; X)/I^{(k-1)}(t; X)$ (and just $I^{(1)}(t; X)$, if $k=1$), so outside the sets where the leading coefficients or the denominators vanish, and where there are ``accidental" root coincidences, each of them has the same number and the same powers of factors in the decomposition into linear factors over $\bc[t]$. It follows that there exists a nonempty, Zariski open subset $S' \subset \Cn$ (which is then open, dense and connected in the Euclidean topology) such that for all $X \in S'$, the number and the sizes of the Jordan blocks in the Jordan normal form and the number of distinct eigenvalues of $\act_X$ is the same.

In the real case, we consider the complexification of $\ip$ and of $\act$, and construct the subset $S' \subset \Cn$ as above, and then use the fact that the intersection of a nonempty, Zariski open subset of $\Cn$ with $\Rn$ is open and dense in $\Rn$ in the Euclidean topology.
\end{proof}

We need the following fact.

\begin{lemma}\label{l:dlambda}
Let $\act$ be an algebraic curvature tensor on $(\Fn, \ip)$, where $\bbf=\br$ or $\bc$. Let $X \in \Fn$ be generic, and let $e \in \Fn$ be an eigenvector of $\act_X$ with the eigenvalue $\mu$. Then on a neighbourhood $U'$ of $X$, there exist a smooth function $\la(Y)$ with $\la(X)=\mu$ and a smooth vector function $f(Y)$ with $f(X)=e$ such that $\act_{Y}f(Y) = \la(Y) f(Y)$, for all $Y \in U'$. Moreover, for any $T \perp X$ we have
\begin{equation} \label{eq:dl}
    2\<\act_e X,T\>=(d\lambda)_{|X}(T)\|e\|^2.
\end{equation}
\end{lemma}
\begin{proof}
The first assertion follows from the fact that for generic $X$, the eigenvalues of $\act_X$ and their corresponding eigenspaces, viewed as the points of the corresponding Grassmannians, depend smoothly (even analytically) on the coordinates of $X$ \cite{K}.

To prove \eqref{eq:dl} we choose a smooth curve $Z(s), \; s \in (-\ve, \ve)$, in $U'$ in such a way that $Z(0)=X$ and $Z'(0)=T$. Then $\act_{Z(s)}f(Z(s)) = \la(Z(s)) f(Z(s))$. Taking the inner product with $f(Z(s))$ and differentiating at $s=0$ we get the required identity.
\end{proof}

We are now in position to prove Theorem~\ref{t:jor} and Theorem~\ref{t:rdp}.

\begin{proof}[Proof of Theorem~\ref{t:jor}]
Suppose $\act$ is a Jordan-Osserman algebraic curvature tensor on $(\Fn, \ip)$. Then any non-null vector $X \in \Fn$ is generic. Moreover, if $e$ is an eigenvector of $\act_X$ with the eigenvalue $\mu$, then, in the notation of Lemma~\ref{l:dlambda}, we obtain $\lambda(Y)=\mu \|Y\|^2 \|X\|^{-2}$, so by \eqref{eq:dl} we get $\<\act_e X,T\>=0$, for all $T \perp X$. It follows that $X$ is an eigenvector of $\act_e$, as required.
\end{proof}

%\begin{remark}\label{rem:sym}
Note that from \eqref{eq:actdef} it follows that if $Y$ is an eigenvector of $\act_X$ with the eigenvalue $\mu_X$, and $X$ is an eigenvector of $\act_Y$ with the eigenvalue $\mu_Y$, then $\mu_X\|Y\|^2=\mu_Y\|X\|^2$.
%\end{remark}

\begin{proof}[Proof of Theorem~\ref{t:rdp}]
We first suppose that a semisimple algebraic curvature tensor $\act$ on $(\Fn,\ip)$ is Osserman. Then the coefficients of the characteristic polynomial $\chi_{X}(t)$ of $\act_X$ are constant on every quadric $\|X\|^2 = c \ne 0$, as are the eigenvalues of $\act_X$. It follows that for all non-null $X \in \Fn$ the Jacobi operator $\act_X$ has the same number $p$ of distinct eigenvalues $\la_k(X)=\mu_k \|X\|^2, \; k=1, \dots, p$. Consider the polynomial $F_X(t)=\prod_{k=1}^p (t-\mu_k \|X\|^2)$. As $\act_X$ is diagonalisable for $X \in s_\act$, we have $F_X(\act_X)=0$. As $s_\act$ has a nonempty interior, this polynomial identity holds for all $X \in \Fn$, which then implies that $\act_X$ is diagonalisable for all non-null $X \in \Fn$. Therefore $\act$ is Jordan-Osserman, hence it satisfies the duality principle by Theorem~\ref{t:jor}.

Conversely, suppose that a semisimple algebraic curvature tensor $\act$ on $(\Fn,\ip)$ satisfies the duality principle. By Lemma~\ref{l:gene}, there is an open subset $U'$ of $s_\act$ consisting of non-null generic vectors. Let $X \in U'$ and let $\mu \in \bbf$ be an arbitrary eigenvalue of $\act_X$. As $\act_X$ is diagonalisable, the orthogonal complement in $\Fn$ to the eigenspace $E(\mu)$ of $\act_X$ is the (direct, orthogonal) sum of all the other eigenspaces of $\act_X$, and so the restriction of $\ip$ to $E(\mu)$ is non-degenerate. We can therefore choose a non-null eigenvector $e \in \Fn$ corresponding to $\mu$. By the duality principle, $X$ is an eigenvector of $\act_e$, so from \eqref{eq:dl} we obtain that $(d\lambda)_{|X}(T)=0$, for all $T \perp X$. From the homogeneity of $\act_X$ we also have $(d\la)_{|X}(X)=2 \mu$, so $(d\lambda)_{|X}(T)=2 \mu\<X,T\>\|X\|^{-2}$, for all $T \in \Fn$. It follows that $d(\|X\|^{-2}\la)_{|X}=0$, hence $d(\|X\|^{2(j-n)}f_j)_{|X}=0$, for all $0 \le j \le n$, where $f_j$ is the coefficient of $t^j$ in the characteristic polynomial $\chi_X(t)$. As all the vectors $Y$ from the open set $U'$ are non-null, generic and belong to the interior of $s_\act$, we obtain that $d(\|Y\|^{2(j-n)}f_j)_{|Y}=0$ for all $Y \in U'$. Hence for some $a_j \in \bbf$ we have $f_j(Y)=a_j \|Y\|^{2(n-j)}$ for all $Y \in U'$, hence for all $Y \in \Fn$. Thus $\act$ is Osserman.
\end{proof}

\end{document}